\documentclass[11pt]{amsart}
\usepackage{hyperref}
\usepackage{anysize} \marginsize{1.3in}{1.3in}{1in}{1in}
\usepackage{amsfonts}

\usepackage[all,cmtip]{xy}
\usepackage{amsmath}
\usepackage{amsthm}
\usepackage{amssymb}
\usepackage{mathrsfs}
\usepackage{enumitem}
\usepackage{tikz-cd}
\usepackage{tikz}
\usepackage[percent]{overpic}

\usepackage[normalem]{ulem}

\newtheorem{thm}{Theorem}[section]

\newtheorem{prop}[thm]{Proposition}
\newtheorem{lemma}[thm]{Lemma}

\newtheorem{cor}[thm]{Corollary}

\theoremstyle{definition}
\newtheorem{defn}[thm]{Definition}

\newtheorem*{definition*}         {Definition}

\newtheorem{eg}[thm]{Example}
\theoremstyle{remark}
\newtheorem*{claim}{Claim}

\newcommand{\EE}{\mathcal{E}}

\newcommand{\DD}{\mathcal{D}}
\newcommand{\HH}{\mathbb{H}}

\newcommand{\Z}{\mathbb{Z}}

\newcommand{\C}{\mathbb{C}}

\newcommand{\arrow}{\rightarrow}

\newcommand{\Res}{\mathrm{Res}}

\renewcommand{\phi}{\varphi}

\def\Res{\operatorname{Res}}

\def\Imag{\operatorname{Im}}

\def\Imag{\operatorname{Im}}

\def\Coh{\operatorname{Coh}}

\def\reduced{\mathrm{red}}

\def\Coh{\mathrm{\mathbf{Coh}}}

\newcommand{\df}{\mathrm{def}}
\newcommand{\an}{\mathrm{an}}

\renewcommand{\bar}[1]{\overline{#1}}

\usepackage{amscd,amssymb,amsmath}
\usepackage{color}

\title{Definable structures on flat bundles}

 \author[B. Bakker]{Benjamin Bakker}
\address{\noindent B. Bakker:  Dept. of Mathematics, Statistics, and Computer Science, University of Illinois at Chicago, Chicago, USA.}
\email{bakker.uic@gmail.com}

\author[S. Mullane]{Scott Mullane}
\address{\noindent S. Mullane: Humboldt-Universit\"at zu Berlin, Institut f\"ur Mathematik,  Unter den Linden 6, 10099 Berlin, Germany} \email{{\tt mullanes@math.hu-berlin.de}}

\def\U{\mathcal{U}}
\def\G{\mathbf{G}}

\def\O{\mathcal{O}}

\def\D{\mathcal{D}}
\def\V{\mathcal{V}}
\def\R{\mathbb{R}}
\def\Q{\mathbb{Q}}

\def\an{\mathrm{an}}
\def\dual{\vee}

\def\X{\mathcal{X}}
\def\XX{\mathscr{X}}
\def\Y{\mathcal{Y}}
\def\D{\mathcal{D}}
\def\B{\mathcal{B}}
\def\E{\mathcal{E}}
\def\EE{\mathscr{E}}
\def\F{\mathcal{F}}
\def\YY{\mathscr{Y}}
\def\DD{\mathscr{D}}
\def\W{\mathcal{W}}
\def\WW{\mathscr{W}}
\def\U{\mathcal{U}}
\def\UU{\mathscr{U}}

\begin{document}
\maketitle
\begin{abstract}
A flat vector bundle on an algebraic variety supports two natural definable structures given by the flat and algebraic coordinates.  In this note we show these two structures coincide, subject to a condition on the local monodromy at infinity which is satisfied for all flat bundles underlying variations of Hodge structures. 
\end{abstract}
\section{Introduction}
Local systems arise frequently in algebraic geometry.  For example, they underlie variations of Hodge structures associated to families of algebraic varieties and encode the monodromy of analytic solutions to algebraic differential equations.

For a smooth algebraic variety $X$, to any complex local system $V_\C$ we may associate a natural vector bundle $V_{\O_{X^\an}}:=\O_{X^\an}\otimes_{\C_{X^\an}}V_\C$ with flat connection $\nabla$ which has a unique algebraic structure $(V_{\O_X},\nabla)$ with regular singularities, and this assignment yields an equivalence of categories called the Riemann--Hilbert correspondence.  With respect to this algebraic structure, the flat sections are highly transcendental.  The purpose of this note is to show that with an assumption on the local monodromy at infinity (which is in particular satisfied for local systems underlying variations of Hodge structures), the change-of-basis matrices between flat and algebraic frames are definable in an o-minimal structure.
   
\begin{defn}
Let $\X$ be a smooth analytic space and $(\Y,\D)$ a log smooth analytic space with $\X=\Y\setminus \D$ and inclusion $j:\X\to \Y$.  We say a complex local system $V_\C$ on $\X$ has norm one eigenvalues at infinity with respect to $(\Y,\D)$ if the local monodromy of $j_*V_\C$ has eigenvalues of complex norm one.

Let $X$ be a complex algebraic variety and $V_\C$ a complex local system on $X^\an$.  We say that $V_\C$ has norm one eigenvalues at infinity if for some resolution $\pi:X'\to X$ and some log smooth compactification $(Y',D')$ of $X'$, the local system $(\pi^\an)^*V_\C$ has norm one eigenvalues at infinity with respect to $(Y'^\an,D'^\an)$. 
\end{defn}

It is easy to see that the pullback along an algebraic map of a local system with norm one eigenvalues at infinity also has norm one eigenvalues at infinity, and in particular the condition on the local mondromy is independent of $(Y',D')$ in the definition.  As mentioned above, for any morphism $f:X\to Y$ of algebraic varieties, any $R^k(f^\an)_*\C_{X^\an}$ which is a local system (they all are if $f$ is smooth projective) has norm one eigenvalues at infinity when pulled back to a smooth base, and therefore satisfies the condition.  

Given an algebraic variety $X$ and a fixed o-minimal structure we can form the associated definable analytic variety $X^\df$, functions on which are definable holomorphic solutions to the equations cutting out $X$ (see section \ref{sec:defstruct} and \cite{bbt} for full details).  Even without assuming $X$ is smooth (or even reduced), we have two sources of vector bundles with flat connection on $X^\df$.  On the one hand, any algebraic vector bundle with flat connection $(E,\nabla)$ naturally yields one $(E,\nabla)^\df$ on $X^\df$ by definabilization; on the other hand, the topology of $X^\df$ (which is the euclidean topology of $X(\C)$) can be trivialized on a definable open cover, so we may form $(\O_{X^\df}\otimes_{\C_{X^\df}}V_\C,\nabla)$ directly.  In general these two objects are not isomorphic (see Example \ref{countereg}), as the transition functions between the flat and algebraic frames may not be definable.  In the case that $V_\C$ has norm one eigenvalues at infinity they are: 

\begin{thm}\label{main}Let $X$ be a complex algebraic space and $V_\C$ a local system on $X^\an$ with norm one eigenvalues at infinity.  Then over $\R_{\an,\exp}$ the definable coherent sheaf $\O_{X^\df}\otimes_{\C_{X^\df}}V_\C$ has a unique algebraic structure $V_{\O_X}$.  Moreover, the connection is algebraic with regular singularities.
\end{thm}

The notation is justified as $(V_{\O_X},\nabla)$ is necessarily the algebraic flat vector bundle corresponding to $V_\C$ via the Riemann--Hilbert correspondence.  The theorem equivalently says that the solutions to the corresponding algebraic differential equation over a definable open set are definable holomorphic functions, or that the flat coordinates of any algebraic section---for example the period integrals of families of algebraic varieties---are definable:
\begin{cor}\label{cor:periods}
Let $f:X\to Y$ be a smooth projective family and $\omega\in H^0(Y,f_*\Omega^p_{X/Y})$ a relative $p$-form.  If $\gamma$ is a section of the degree $p$ homology local system $R_p(f^\an)_*\Q_{X^\an}$ over a definable open set $U$, then the integral
\[\int_{\gamma_u}\omega_u\]
is a $\R_{\an,\exp}$-definable function on $U$.
\end{cor}
\begin{proof}The algebraic structure on the flat vector bundle associated to $V_\C=R^p(f^\an)_*\C_{X^\an}$ is algebraic de Rham cohomology $V_{\O_X}=R^pf_*\Omega^\bullet_{X/Y}$, equipped with the Gauss--Manin connection.
Integration $\int_{\gamma}$ is a section of $V_\C^\dual$, hence a definable section of $\O_{X^\df}\otimes_{\C_X}V^\vee_\C$, while $\omega$ is an algebraic section of $V_{\O_X}$.  As $V_\C$ has norm one eigenvalues at infinity, $(V_{\O_X})^\df\cong \O_{X^\df}\otimes_{\C_X}V^\vee_\C$, and the claim follows.
\end{proof}

\subsection*{Notation}  All of our analytic spaces, definable analytic spaces, and algebraic spaces will be over $\C$ and assumed to be separated.  In the algebraic category, our algebraic spaces will be in addition, of finite type.  We use the symbol ``$\Subset$" to mean ``relatively compact open subspace", with ``subspace" to be interpreted as either an analytic subspace or a definable analytic subspace depending on context.

\section{Definable structures on compact analytic spaces}\label{sec:defstruct}
Throughout this section we work over an o-minimal structure containing $\R_\an$.  Recall that this means any overconvergent real analytic function on a euclidean ball is definable.  The main result is to show that any compact analytic space admits a unique definable analytic structure, and that the two categories of coherent sheaves are naturally equivalent.  We refer to \cite{bbt} for details on definable analytic spaces.  

Recall that a definable topological space $\X$ is a topological space equipped with an equivalence class of finite atlases by definable open subsets of euclidean space with definable transition functions.  The definable site $\underline{\X}$ of a definable topological space $\X$ is the category of definable open subsets whose coverings are finite (definable) open coverings and we refer to sheaves on $\underline \X$ as just sheaves on $\X$. A definable analytic space $\X$ is a definable topological space $|\X|$ with a sheaf of local $\C$-algebras $\O_\X$ on $|\X|$ which is on a covering isomorphic to the zero locus $\V=V(f_1,\ldots,f_k)\subset \U\subset\C^n$ of finitely many definable holomorphic functions $f_1,\ldots,f_k:U\to\C$ on a definable open $\U\subset \C^n$ equipped with the sheaf $\O_{\C^n}/(f_1,\ldots,f_k)\O_{\C^n}|_\V$ on $\V$ as a definable topological space.  Here $\O_{\C^n}$ is the sheaf of definable holomorphic functions on $\C^n$ as a definable topological space in the obvious way.  We say a definable analytic space $\X$ is compact if the ordinary underlying topological space is compact.

There is a natural analytification functor
    \[(-)^\an:(\mbox{definable analytic spaces})\to (\mbox{analytic spaces}).\]  The underlying topological space of $\X^\an$ is the ordinary topological space underlying $\X$ and $\O_{\X^\an}$ is in this case (since we work over an o-minimal structure containing $\R_\an$) the sheafification of $\O_\X$ in the euclidean topology.  Likewise there is a natural analytification functor
    \[(-)^\an:\Coh(\X)\to\Coh(\X^\an)\]
    which is just sheafification in the euclidean topology.  We say objects or morphisms in the essential image of either analytification functor are definabilizable.  
    
   Both analytification functors are faithful \cite{bbt}, but are in general far from equivalences.  In the case of compact spaces, however, we have the following:
    \begin{prop}\label{cpt}  \hspace{1in}
\begin{enumerate}
    \item\label{cpt1} The restriction of the analytification functor to the full subcategories of compact spaces
    \[(-)^\an:(\mbox{compact definable analytic spaces})\to (\mbox{compact analytic spaces})\]
    is an equivalence of categories.
    \item\label{cpt2}   For any compact definable analytic space $\X$, the analytification functor
    \[(-)^\an:\Coh(\X)\to\Coh(\X^\an)\]
    is an equivalence of abelian categories.
\end{enumerate}
\end{prop}

From classical GAGA \cite{Serre} (see \cite[Th\'eor\`eme 5.10]{toe} for the statement for algebraic spaces) we immediately deduce the following:
\begin{cor}
Let $X$ be a proper complex algebraic space.  Then the three functors
\[\begin{tikzcd}
\Coh(X)\arrow{rd}[swap]{(-)^\an}\ar{rr}{(-)^\df}&&\Coh(X^\df)\ar{ld}{(-)^\an}\\
&\Coh(X^\an)&
\end{tikzcd}
\]
are equivalences of abelian categories.
\end{cor}

    We prove Proposition \ref{cpt} via two slightly more general lemmas.  The first establishes the fullness of the analytification functors up to restricting to a relatively compact open subspace.  
    
    \begin{lemma}\label{full}  Let $\XX$ be an analytic space, $\X,\Y$ definable analytic spaces, and $\E,\F$ coherent sheaves on $\X$.
\begin{enumerate}
    \item\label{full1} Every point $x\in \XX$ admits a definabilizable relatively compact open neighborhood $x\in \UU\Subset \XX$.
    \item\label{full2} For any morphism $\phi:\X^\an\to \Y^\an$ and any definable relatively compact open $\U\Subset \X$, the restriction $\phi|_{\U^\an}:\U^\an\to \Y^\an$ is definabilizable. 
    \item\label{full3} For any morphism $\phi:\E^\an\to \F^\an$ and any definable relatively compact open $\U\Subset \X$,  the restriction $\phi|_{\U^\an}:\E^\an\to \F^\an$ is definabilizable.
\end{enumerate}
\end{lemma}
\begin{proof}
For part \eqref{full1} we may assume $\XX$ is locally the zero locus $V(I)\subset \WW\subset(\C^n)^\an$ of a finitely generated ideal $I=(f_1,\ldots,f_m)\subset \O_{(\C^n)^\an}(\WW)$ of holomorphic functions on an open subset $\WW\subset (\C^n)^\an$.  The topology of $(\C^n)^\an$ has a basis by relatively compact definable open subsets (for example euclidean balls) and for each definable open $(\W')^\an\Subset \WW\subset(\C^n)^\an$, the restrictions $f_i|_{\W'^\an}$ are $\R_\an$-definable and hence $V(I)|_{\W'^\an}$ is definabilizable open and $V(I)|_{\W'^\an}\Subset V(I)\subset \XX$.  This proves part \eqref{full1}.

For parts \eqref{full2}, \eqref{full3}, by the faithfulness of analytification the claim is local on $\U$.  Observe that by part \eqref{full1} and the relative compactness of $\U$, any collection of open subspaces of $\X^\an$ covering $\U^\an$ can be refined by finitely many open $\X_i\subset \X$ and $\X_i'\Subset \X_i$ such that the $\X_i'$ cover $\U$.  It then suffices to prove the claim replacing $\U\subset \X$ with $\X_i'\subset \X_i$.  

Applying this observation to $\phi^{-1}(\Y_i^\an)$ for a definable cover $\Y_i$ of $\Y$, we may assume $\Y$ is a local model $V(I)\subset \W\subset \C^n$; applying it to a definable cover of $\X$, we may assume $\X$ is a local model $V(I')\subset \W'\subset\C^m$, and moreover that $f$ is given by a holomorphic map $g:\W'^\an\to \W^\an$ with $g^\sharp I\subset I'$.  We then have that $g|_{\U}$ is $\R_\an$-definable, and $f$ is the analytification of the induced morphism $\X\to \Y$.  This yields part \eqref{full2}

Likewise, by passing to a covering of $\X$ we may first assume $\X$ is a local model $V(I)\subset \W\subset\C^n$, next that $\E$ and $\F$ are both quotients of $\O_\W^k$, and finally that $\phi:\E^\an\to \F^\an$ lifts to a morphism $\O_{\W^\an}^k\to\O_{\W^\an}^k$.  We therefore reduce to $\X=\W\subset\C^n$ a definable open subset and $\E=\F=\O_\X^k$.  The morphism $\phi$ is given by a matrix of holomorphic functions on $\X^\an$, which when restricted to $\U^\an$ are $\R_\an$-definable.

\end{proof}

The next lemma handles the essential surjectivity of the analytification functors.
\begin{lemma}\label{ess}Let $\XX$ be an analytic space, $\X$ a definable analytic space, $\EE$ a coherent sheaf on $\X^\an$.
\begin{enumerate}
    \item\label{ess1} Any relatively compact open $\UU\Subset\XX$ is relatively compact in a definabilizable relatively compact open $\UU'\Subset \XX$.
    \item\label{ess2} For any definable relatively compact $\U\Subset \X$, the restriction $\EE|_{\U^\an}$ is definabilizable.
\end{enumerate}
\end{lemma}
\def\cpt{\mathrm{cpt}}
\begin{proof}
For part \eqref{ess1}, as in Lemma \ref{full} we may take finitely many open definabilizable $\XX_i\subset\XX$ with definabilizable $\XX_i'\Subset \XX_i$ such that the $\XX_i'$ cover $\UU$.  The proof is completed by inductively applying the following:
\begin{claim}
For $i=1,2$, let $\XX_i\subset \XX$ be a definabilizable open subspace with definabilizable $\XX_i'\Subset\XX_i$. Then there exists a definabilizable $\XX''$ such that $\XX_1'\cup\XX_2'\Subset \XX''\Subset \XX_1\cup\XX_2$.
\end{claim}
\begin{proof}
By Lemma \ref{full}\eqref{full2}, we may suppose $\XX_i'\subset\XX_i$ is the analytification of $\X_i'\subset \X_i$.  We have that $\XX_1'\cap \XX_2'\Subset\XX_1\cap \XX_2$, hence there is a definable $\B\Subset \X_1$ for which $\XX_1'\cap\XX_2'\Subset \B^\an\Subset \XX_1\cap \XX_2$ under the identification $\X_1^\an\cong \XX_1$.  By Lemma \ref{full}\eqref{full2}, the composition $\B^\an\subset \XX_1\cap\XX_2\to \XX_2\cong \X_2^\an$ is the analytification of an open immersion $j:\B\to \X_2$.  Let $\Y\Subset \X_2$ be a definable open subspace such that $\X_2'\Subset \Y$ and $\Y^\an\cap \XX_1'=j(\B)^\an\cap \XX_1'$ under the identification $\X_2^\an\cong \XX_2$, which is possible since $\XX_2'\setminus j(B)^\an$ is relatively compact in $\XX_2\setminus\overline{\XX_1'}$.  Then we have a natural open immersion $(\X_1'\cup_\B\Y)^\an\to \XX$ with the required properties.    

\end{proof}
For part \eqref{ess2}, we may likewise refine any cover of $\X$ by finitely many open definable $\X_i\subset \X$ with $\X_i''\Subset \X_i'\Subset \X_i$ such that the $\X''_i$ cover $\U$.  Then provided each $\EE|_{\X_i^\an}$ is definabilizable, the gluing maps on $(\X_i''\cap\X_j'')^\an$ will be definabilizable by Lemma \ref{full}\eqref{full3} (uniquely by the faithfulness of analytification), and it will follow that $\EE|_{\U^\an}$ is definabilizable.  We may therefore assume we have a presentation
\[\O_{\X^\an}^{m}\xrightarrow{\phi}\O_{\X^\an}^{n}\to \EE\to 0.\]
By Lemma \ref{full}, $\phi|_{\U}$ is definabilizable, and as analytification is exact it follows that $\EE|_{\U^\an}$ is definabilizable.

\end{proof}

\begin{proof}[Proof of Proposition \ref{cpt}]
We already know faithfulness.  Lemma \ref{full} yields the fullness, and Lemma \ref{ess} the essential surjectivity.
\end{proof}

\section{Proof of Theorem \ref{main}}

In this section we prove Theorem \ref{main}.  The main input is to show that the Deligne canonical extension can be formed definably, provided the monodromy has norm one eigenvalues at infinity.  We work throughout over an o-minimal structure containing $\R_{\an,\exp}$.

On a smooth analytic space $\XX$, the categories of complex local systems on $\XX$ and locally free $\O_\XX$-modules with flat connection are naturally equivalent via $V_\C\mapsto (V_{\O_\XX},\nabla)$.  Suppose $(\YY,\DD)$ is a log smooth pair with $\YY\setminus \DD=\XX$ and that $(V_{\O_\XX},\nabla)$ is a locally free $\O_\XX$-module with flat connection.  According to \cite{deligneext}, for any choice of section $\Sigma\subset \C$ of the complex exponential $e(\cdot)=e^{2\pi i\cdot}:\C\to\C^*$ there is an extension $( V_{\O_\YY},\nabla)$ of $(V_{\O_\XX},\nabla)$ to $\YY$ as a locally free $\O_{\YY}$-module with flat logarithmic connection, unique up to unique isomorphism, such that the eigenvalues of $\Res\nabla$ are contained in $\Sigma$ pointwise.  Here by an extension $\bar F$ of a locally free sheaf $F$ on $\XX$ we mean a locally free sheaf together with an isomorphism $\bar F|_{\XX}\cong F$, and we further ask the isomorphism to be compatible with the connection.  

For $X$ a smooth algebraic variety, $V_{\O_{X^\an}}$ a locally free $\O_{X^\an}$-module with flat connection, $\Sigma$ a choice of section of $e$, and $(Y,D)$ a choice of log smooth compactification of $X$, we therefore obtain by classical GAGA a functorial algebraic structure $(V_{\O_{Y}},\nabla)$ on $(V_{\O_{Y^\an}},\nabla)$.  The fact that the connection of $V_{\O_{Y^\an}}$ has logarithmic singularities means in particular that it induces a connection on the associated meromorphic bundle and therefore that the connection is algebraic (with logarithmic singularities) with respect to the algebraic structure $V_{\O_{Y}}$.  We therefore obtain an algebraic structure $(V_{\O_X},\nabla)$ with regular singularities on $(V_{\O_{X^\an}},\nabla)$.  One further shows (see e.g. \cite[IV \S 5]{borel}) that flat sections extend meromorphically, and this implies that analytification yields an equivalence of categories between $\O_X$-modules with flat connection and regular singularities and complex local systems on $X^\an$.

Let $\X$ be a definable analytic space and $V_\C$ a complex local system on $\X^\an$.  The restriction of $V_\C$ to $\X$ naturally sheafifies to $V_\C$ since $\X$ has a definable cover by simply connected open subsets by definable triangulation \cite[Chapter 8, (2.9)]{Vdd}.  We denote the restriction by $V_\C$ as well.  Then $V_{\O_\X}:=\O_{\X}\otimes_{\C_\X}V_\C$ is naturally a definable coherent sheaf with flat connection $\nabla$, and naturally analytifies to $(V_{\O_{\X^\an}},\nabla)$.

\begin{prop}\label{del}Let $(\Y,\D)$ be a definable analytic log smooth pair, and let $\U\Subset \Y$ be a relatively compact definable open subspace.  Set $\X=\Y\setminus \D$, $\W=\U\setminus \D$.  Let $V_\C$ be a complex local system on $\X^\an$ with norm one eigenvalues at infinity with respect to $(\Y,\D)$.  
\begin{enumerate}
    \item\label{del1} There is an extension $(V_{\O_\U},\nabla)$ of $(V_{\O_\W},\nabla)$ as a locally free sheaf with logarithmic connection whose residue has eigenvalues in $\Sigma$. 
    \item\label{del2} Any two such extensions of $(V_{\O_\X},\nabla)$ are  isomorphic by a unique isomorphism (as extensions and compatibly with the connection) over $\U$.
\end{enumerate}
 
\end{prop}
Note that the extension $(V_{\O_\U},\nabla)$ provided by the proposition analytifies to the Deligne canonical extension $(V_{\O_{\U^\an}},\nabla)$ as an extension of $(\O_{\W^\an}\otimes_{\C_{\W^\an}}V_\C,\nabla)$.
\begin{proof}We first treat the local case.
\begin{lemma}\label{poly}
Let $\X=(\Delta^*)^k\times\Delta^\ell\subset \Y=\Delta^{k+\ell}$ with their standard structures as definable complex analytic spaces.  Then for any local system $V_\C$ on $\X$ whose monodromy has norm one eigenvalues, $V_{\O_\X}$ admits an extension to $\Y$ as a locally free sheaf with logarithmic connection $\nabla$ in the definable analytic category and the residues of $\nabla$ have eigenvalues in $\Sigma$.
\end{lemma}
\begin{proof}
Choose a basepoint $x\in \X$ and consider the universal cover $\pi:\HH^k\times\Delta^\ell\to(\Delta^*)^k\times\Delta^\ell$ given by $\pi(z_1,\ldots,z_k,q_1,\ldots,q_\ell)=(e(z_1),\ldots,e(z_k),q_1,\ldots,q_\ell)$ where $e(z)=e^{2\pi i z}$.  Note that for any open bounded vertical strip $S\subset \HH$ with its standard definable structure, the restriction $\pi|_F$ to $F=S^k\times \Delta^\ell$ is $\R_{\an,\exp}$-definable.  Set $\tilde \X=\HH^k\times\Delta^\ell$.  A local system $V_\C$ on $\X$ corresponds to a monodromy representation $\rho:\pi_1(\X,x)\to \mathrm{GL}(V_{\C,x})$.  We canonically trivialize $\pi^*V_\C\cong \tilde \X\times V_{\C,x}$ with $\pi_1(\X,x)$-action $\gamma:(z,v)\mapsto (\gamma z,\rho(\gamma)v)$.   

Let $T_i$ be the monodromy around the $i$th punctured disk factor.  By taking the Jordan decomposition, we see that there are unique $N_i\in\frak{gl}(V_{\C,x})$ with eigenvalues in $\Sigma$ such that $T_i=\exp(2\pi i N_i)$.  Let $N_i=N^{ss}_i+N^u_i$ be the Jordan decomposition of $N_i$.  For any $v_x\in V_{\C,x} $, $\tilde v=\exp(2\pi i \sum_iz_i\otimes N_i)v_x$ is a $\pi_1(\X,x)$-invariant section of $\pi^*(\O_{\X^\an}\otimes_{\C_{\X^\an}}V_\C)$ and therefore descends to a section $v$ of $\O_{\X^\an}\otimes_{\C_{\X^\an}}V_\C$.  We claim that $v$ is in fact a section of $\O_\X\otimes_{\C_\X}V_\C$.  This may be checked on a definable open cover provided by maps of the form $\pi|_F$ as above, and $\exp(2\pi i \sum_j z_j\otimes N_j^u)$ is polynomial in the $z_j$ while $\exp(2\pi i \sum_j z_j\otimes N_j^{ss})$ is $\R_{\an,\exp}$-definable on $F$ since the $N_i^{ss}$ is real by the condition on the eigenvalues of $T_i$.

The assignment $v_x\mapsto v$ provides an isomorphism $\O_{\X}\otimes_\C V_{\C,x}\to \O_{\X}\otimes_{\C_{\X}}V_\C$ and therefore an extension $\O_{\Y}\otimes_\C V_{\C,x}$.  We see that $\Res\nabla=N_j$ along the $j$th boundary divisor as usual.
\end{proof}
By Lemma \ref{full}\eqref{full1} and the relative compactness of $\U$ we may take finitely many open subspaces $\Y_i\subset \Y$ with $\Y_i''\Subset \Y_i'\Subset \Y_i$ such that the $\Y''_i$ cover $\U$.  Then the existence statement for $\Y_i'\subset \Y_i$ coupled with the uniqueness for $\Y_i''\cap \Y_j''\subset \Y_i'\cap \Y_j'$ for each $i,j$ will yield part \eqref{del1}, while the uniqueness on $\Y_i'\subset \Y_i$ for each $i$ yields part \eqref{del2}.  We can take such a cover such that each pair $\X\cap \Y_i\subset \Y_i$ is isomorphic to $(\Delta^*)^{k_i}\times\Delta^{\ell_i}\subset \Delta^{k_i+\ell_i}$, and moreover such that the $\Y_i''$ and $\Y_i'$ are identified with concentric polydisks.

We therefore restrict to the case $\U\subset \Y$ are concentric polydisks.  Lemma \ref{poly} implies the existence of $V_{\O_\Y}$ (and therefore $V_{\O_\U}$).  Given two extensions $\E,\F$, by the uniqueness in the analytic case there is an isomorphism $\phi:\E^\an\xrightarrow{\cong} \F^\an$ of extensions which is compatible with the connections, and by Lemma \ref{cpt}\eqref{cpt1} it follows that $\phi|_\U$ is definabilizable, whence the claim.

\end{proof}

\begin{proof}[Proof of Theorem \ref{main}]
By definable GAGA \cite{bbt} the algebraic structure is unique if it exists.  Observe that if $V_{\O_{X^{\df}}}$ has an algebraic structure, then so does $F^\df\otimes_{\C_{X^\df}}V_\C\cong F^\df\otimes_{\O_{X^\df}} V_{\O_{X^\df}}$ for any coherent sheaf $F$ on $X$.

We now show that $V_{\O_{X^\df}}$ is algebraic.  By definable GAGA we may assume $X$ is affine, since if the claim is true in this case then the gluing maps associated to an affine (\'etale) cover will be algebraic.  Assume first that $X$ is smooth and let $(Y,D)$ be a log smooth compactification.  By Proposition \ref{del}, $(V_{\O_{X^\df}},\nabla)$ admits an extension $(V_{\O_{Y^\df}},\nabla)$.  By classical GAGA, $(V_{\O_{Y^\df}})^\an$ has an algebraic structure $V_{\O_{Y}}$, and by Proposition \ref{cpt} we have $(V_{\O_Y})^\df\cong V_{\O_{Y^\df}}$.  If $X$ is possibly non-reduced but $X^\reduced$ is smooth, then the inclusion $i:X^\reduced\to X$ admits a finite section $r:X\to X^\reduced$.  It follows by the above that $r_*(V_{\O_{X^\df}})\cong r_*\O_{X^\df}\otimes_{\C_{X^\df}}V_\C$ is algebraic, and by definable GAGA that the $r_*\O_{X^\df}$-module structure is algebraic as well.  Thus, $V_{\O_{X^\df}}$ is algebraic.  Finally, for arbitrary $X$, by blowing up along reduced centers we may produce a proper map $\pi:Y\to X$ for which $Y^\reduced$ is smooth and which is dominant on a dense Zariski open set $U$ of $X$.  Let $Z\subset X$ be the reduced complement of $U$ and $X'$ the image of $\pi$.  Then for a sufficiently large thickening $Z'$ of $Z$, the pushout $P$ of the diagram 
\[\begin{tikzcd}
Z'\times_XX'\ar[d]\ar[r]&X'\\
Z'&
\end{tikzcd}\]
has a natural proper dominant map $f:P\to X$.  By Noetherian induction $\O_{Z'^\df}\otimes_{\C_{Z'^\df}}(V_\C|_{Z'^\df})$ is algebraic while by definable GAGA we have that $\O_{X'^\df}\otimes_{\C_{X'^\df}}V_\C\subset \pi_*(\O_{Y^\df}\otimes_{\C_{Y}}\pi^{-1}V_\C)$ is algebraic.  The pushout $\O_{P^\df}\otimes_{\C_{P^\df}}f^{-1}V_\C$ is therefore also algebraic, as is $V_{\O_{X^\df}}\subset f_*(\O_{P^\df}\otimes_{\C_{P^\df}}f^{-1}V_\C)$ by definable GAGA.  

It remains to show that the connection is algebraic with regular singularities.  For the first claim, we may assume $X$ is affine and that $V_{\O_X}$ is free as an $\O_X$-module.  Then for any algebraic section $s$ of $V_{\O_{X^\df}}$, $\nabla s$ is a definable section of $V_{\O_{X^\df}}\otimes_{\O_{X^\df}}\Omega_{X^\df}\cong (V_{\O_X}\otimes_{\O_X}\Omega_X)^\df$, hence algebraic by definable GAGA.  Finally, from the construction (in particular Proposition \ref{del}) it is clear that the singularities of the connection are regular.
\end{proof}

We conclude this section with an example (see \cite[Example 3.2]{bbt}) which shows Theorem \ref{main} is false without the condition on the monodromy at infinity.

 \begin{eg}\label{countereg}
Let $X=\G_m$ with coordinate $q$.  Let $\alpha\in\C$ and consider the rank one $\C$-local system $V$ on $X^\an$ with multiplicative monodromy $\lambda=e^{2\pi i \alpha}$.  If $\mathcal{F}=\O_{X^\df}\otimes_{\C_{X^\df}}V$ were algebraic, it would necessarily be trivial; we claim that if $\alpha\notin\R$ then $\mathcal{F}$ will not be trivial in any o-minimal structure.  A trivializing section is of the form $v\otimes f$ for a nowhere zero multivalued holomorphic function $f$ on $\C^*$ with monodromy $\lambda$.  After multiplying by some power $q^n$, we may assume $f=e^{\alpha\log q+g(q)}$ for a holomorphic function $g:\C^*\to\C$.  As $f'/f=\alpha q^{-1}+g'(q)$ is single-valued and definable, it cannot have essential singularities at $0$ or $\infty$ (or else it would have infinite fibers), and therefore $g$ is algebraic---in particular, a polynomial in $q,q^{-1}$.  But restricting to positive real $q$, we have that
\[\{q\in\R_{>0}\mid f(q)\in\R\}=\{q\in\R_{>0}\mid (\Imag g)(q)+(\Imag \alpha)\log q\in \pi\Z\}\]
is definable, which is only the case if $\Imag g$ is constant and $\Imag \alpha=0$.

\end{eg}

\bibliography{biblio.flatdef}
\bibliographystyle{plain}
\end{document}